\documentclass[12pt]{article} 

\usepackage{amsmath}
\usepackage{amssymb}
\usepackage{amsthm}
\usepackage[all,cmtip]{xy}
\usepackage{fancyhdr}
\usepackage{euscript}
\usepackage{setspace}
\usepackage{graphicx,epsfig}
\usepackage{palatino, url, multicol}
\usepackage{mathrsfs} 
\usepackage{enumerate}   
\usepackage{verbatim} 

\theoremstyle{definition}

\newtheorem{theorem}{Theorem}[section]

\newtheorem{lemma}[theorem]{Lemma}

\newtheorem{proposition}[theorem]{Proposition}

\newtheorem{corollary}[theorem]{Corollary}

\newtheorem{remark}[theorem]{Remark}

\newcommand{\IR}{\hbox{$\mathbb{R}$}}

\newcommand{\IZ}{\hbox{$\mathbb{Z}$}}
\newcommand{\IC}{\hbox{$\mathbb{C}$}}
\newcommand{\IQ}{\hbox{$\mathbb{Q}$}}

\newcommand{\abs}[1]{\hbox{$\left| {#1} \right|$}}

\renewcommand{\hat}{\widehat}

\def\it{\itshape}

\def\tt{\texttt}
\def\bf{\textbf}

\def\IR{\mathbb{R}}
\def\IQ{\mathbb{Q}}

\def\I1{\mathbb{1}}
\def\IC{\mathbb{C}}

\def\ccK{\mathscr{K}}

\def\ccT{\mathscr{T}}

\def\limx0{\lim_{x \to 0}}

\def\intxyleq1{\underset{\| x - y  \| \leq 1}{\int}}
\def\intxygeq1{\underset{\| x - y  \| \geq 1}{\int}}
\def\intxizetaleq1{\underset{\| \xi - \zeta  \| \leq 1}{\int}}
\def\intxizetageq1{\underset{\| \xi - \zeta \| \geq 1}{\int}}

\def\tab{\hskip 1mm}

\def\tab{\hspace{.1pc}}
\def\ttab{\hspace{1pc}}

\newcounter{hours}
\newcounter{minutes}
\newcommand\printtime{%
  \setcounter{hours}{\the\time/60}%
  \setcounter{minutes}{\the\time-\value{hours}*60}%
  \ifthenelse{\value{hours} > 12}
     {
       \setcounter{hours}{\value{hours}-12}%
       \thehours:\theminutes \ p.m.                
     }
     {
       \thehours:\theminutes \ a.m.                
     } 
}

\def\putdate{{\tt Compiled on \the\month-\the\day-\the\year \ at\printtime} \\}

\begin{document} 

\title{Arithmetic Diophantine Approximation for Continued Fractions-like Maps on the Interval}
 \author{Avraham Bourla\\
   Department of Mathematics\\
   Saint Mary's College of Maryland\\
   Saint Mary's City, MD, 20686\\
   \texttt{abourla@smcm.edu}}
 \date{\today}
 \maketitle
\begin{abstract}
\noindent We establish arithmetical properties and provide essential bounds for bi-sequences of approximation coefficients associated with the natural extension of maps, leading to continued fraction-like expansions. These maps are realized as the fractional part of M$\operatorname{\ddot{o}}$bius transformations which carry the end points of the unit interval to zero and infinity, extending the classical regular and backwards continued fractions expansions.
\end{abstract}

\section{Introduction and preliminaries}{}

\subsection{Introduction}

Given a real number $r$ and a rational number, written as the unique quotient $\frac{p}{q}$ of the relatively prime integers $p$ and $q>0$, our fundamental object of interest from diophantine approximation is the \bf{approximation coefficient} $\theta(r,\frac{p}{q}) := q^2\abs{r-\frac{p}{q}}$. Small approximation coefficients suggest high quality approximations, combining accuracy (reflected by the error of approximation) with simplicity (reflected by a small denominator). Adding integers to fractions does not change their denominators, hence $\theta(r,\frac{p}{q}) = \theta(r - \lfloor r \rfloor ,\frac{p}{q} - \lfloor r \rfloor)$, where the \bf{floor} $\lfloor r \rfloor$ of $r$ is the largest integer smaller than or equal to $r$, allowing us to restrict our attention to the unit interval. We expand an irrational initial seed $x_0 \in (0,1) - \IQ$ as a regular continued fraction or \bf{RCF} and obtain the unique infinite sequence $\{b_n\}_1^\infty$ of partial quotients or \bf{0-digits} of $x_0$. The sequence of rational numbers 
\[\frac{p_0}{q_0} := \frac{0}{1}, \hspace{1pc} \frac{p_n}{q_n} := [b_1,...,b_n]_0 = \frac{1}{b_1 + \frac{1}{b_2 + ... + \frac{1}{b_n}}}, \ttab n \ge 1, \]
called the \bf{convergents} of $x_0$ are also uniquely determined (the reason for the subscript $[\cdot]_0$ will become clear later).\\ 

\noindent For all $n \ge 0$, it is well known \cite[Theorem 4.6]{Burger} that $\abs{x_0 - \frac{p_n}{q_n}} < \frac{1}{q_n{q_{n+1}}} < \frac{1}{q_n^2}$, so that each convergent has an associated coefficient which is smaller than one. Conversely, Legendre \cite[Theorem 5.12]{Burger} proved that if $\theta(x_0,\frac{p}{q}) < \frac{1}{2}$ then $\frac{p}{q}$ is a convergent of $x_0$. The question of existence for such high quality approximations was settled in 1891 by Hurwitz \cite[theorem 5.1.4]{DK}, who proved there exist infinitely many pairs of integers $p$ and $q$, such that $\theta(x_0,\frac{p}{q}) < \frac{1}{\sqrt{5}}$, where this inequality is sharp. We conclude that all irrational numbers enjoy infinitely many rational approximations with associated coefficients of less than the \bf{Hutwitz Constant} $\frac{1}{\sqrt{5}}$, and that all of these quality approximations must belong to the sequence of RCF convergents. Define the approximation coefficient associated with each convergent of $x_0$ by
\[ \theta_n(x_0) = \theta_n  := \theta\bigg(x_0,\frac{p_n}{q_n}\bigg) = q_n^2\abs{r - \frac{p_n}{q_n}}\]
and refer to the sequence $\{\theta_n\}_0^\infty$ as the \bf{sequence of approximation coefficients}.\\   

\noindent Much work has been done with this sequence, from its inception in the classical era, till the more recent excursions \cite{Bourla, BM, DK, JP, Tong}. These reveal elegant internal structure as well as simple connections to the sequence of 0-digits. The introduction section of \cite{Bourla} provides a survey of these results, whereas a more thorough treatment can be found in \cite{DK}. Furthermore, the essential lower bounds for this sequence determine how well can irrational numbers be approximated using rational numbers and lead to the construction of the Lagrange Spectrum \cite{CF}. Our goal is show that the RCF theory extends well to the classes of continued fraction-like expansions, first introduced by Hass and Molnar in \cite{HM, HM2}. 

\subsection{The dynamics for regular and backwards continued fractions}

From a dynamic point of view, the regular continued fraction expansion, or \bf{0-expansion}, is a concrete realization of the symbolic representation of irrational numbers in the unit interval under the iterations of the \bf {Gauss Map} 
\[T_0: [0,1) \to [0,1), \hspace{1pc} T_0(x) := \frac{1}{x} - \bigg\lfloor \frac{1}{x} \bigg\rfloor, \hspace{1pc} T_0(0) := 0.\]
This map is the fractional part of the homeomorphism $A_0:(0,1) \to (1, \infty), \tab x \mapsto \frac{1}{x}$, which, in turn, extends to the M$\operatorname{\ddot{o}}$bius transformation $\hat{A}_0:\hat{\IC} \to \hat{\IC}, \tab z \mapsto \frac{1}{z}$ mapping $[0,1]$ bijectively to $[1,\infty]$ in an orientation reversing manner. The Gauss map is both invariant and ergodic with respect to the probability Gauss measure on the interval $\mu_0(E) := \frac{1}{\ln{2}}\int_E\frac{1}{1-x}dx$.\\

\noindent Another well known continued fraction theory is the backwards continued fractions (\bf{BCF}) expansion or {\bf 1\bf{-expansion}}, stemming from the \bf{Renyi Map} 
\[T_1: [0,1) \to [0,1), \hspace{1pc} T_1(x) := \frac{1}{1-x} - \bigg\lfloor \frac{1}{1-x} \bigg\rfloor. \] 
This map is the fractional part of the homeomorphism $A_1:(0,1) \to (1, \infty), \tab x \mapsto \frac{1}{1-x}$, which extends to the M$\operatorname{\ddot{o}}$bius transformation $\hat{A}_1:\hat{\IC} \to \hat{\IC}, \tab z \mapsto \frac{1}{1-z}$, mapping $[0,1]$ bijectively onto $[1,\infty]$ in an orientation preserving manner. The Renyi map is invariant and ergodic with respect to the infinite measure $\mu_1(E) := \int_E\frac{1}{x}dx$ on the interval.\\ 

\noindent Letting $m \in \{0,1\}$, we extract the sequence of digits for the $m$-expansion of a real number $x_0 \in (0,1)$ using the following iteration process:
\begin{enumerate}
\item Set $n := 1$.
\item If $x_{n-1} = 0$, write $x_0 = [b_1,...,b_{n-1}]_m$ and exit.
\item Let $b_n :=  \lfloor A_m(x_{n-1}) \rfloor$ and 
\[x_n = A_m(x_{n-1}) - a_n = A_m(x_{n-1}) - \lfloor A_m(x_{n-1}) \rfloor = T_m(x_{n-1}).\] 
Increase $n$ by one and goto step 2.
\end{enumerate}
 For instance
\[[1,1,2]_0 = \frac{1}{1 + \frac{1}{1 + \frac{1}{2}}} = \frac{1}{1 + \frac{2}{3}} = \frac{3}{5} = 1 - \frac{2}{5} = 1 - \frac{1}{2 + 1 - \frac{1}{2}} = [2,2]_1.\]

\begin{remark}\label{m=k=1}
The fact that $T_1$ has as an indifferent fixed point at the origin forces any absolutely continuous invariant measure to be infinite \cite{Thaler}. This deficiency helps explain why the BCF theory did not gain nearly as much attention as its RCF cousin, even though it sometimes leads to quicker expansions, as seen in the example above. For more details about the BCF expansion, refer to \cite{HBackwards}.
\end{remark}

\subsection{Gauss-like and Renyi-like continued fractions}{}

The rest of this section is a paraphrased summery of previous work due to Haas and Molnar \cite{HM, HM2}. In general, the fractional part of M$\operatorname{\ddot{o}}$bius transformation which map $[0,1]$ onto $[0,\infty]$ leads to expansion of real numbers as continued fractions. To characterize all these transformations, we recall that M$\operatorname{\ddot{o}}$bius transformations are uniquely determined by their value on three distinct points. Thus, we will need to introduce a parameter for the image of an additional point besides 0 and 1, which we will naturally take to be $\infty$. Since our maps fix the real line, the image of $\infty$, denoted by $-k$, can take any value within the set of all negative real numbers. After letting $m \in \{0,1\}$ equal zero or one for orientation reversing and preserving transformations respectively, we conclude that all such transformations are derived as extension of the homeomorphisms 
\[A_{(m,k)}:(0,1) \to (0,\infty), \ttab x \mapsto \frac{k(1-m-x)}{x-m}, \ttab k>0\] 
from the open unit interval to its closure. The maps $T_{(m,k)}: [0,1) \to [0,1), \tab 0 \mapsto 0$,
\begin{equation}\label{T}
T_{(m,k)}x = A_{(m,k)}(x) - \lfloor A_{(m,k)}(x) \rfloor = \frac{k(1-m-x)}{x-m} - \bigg\lfloor\frac{k(1-m-x)}{x-m}\bigg\rfloor,\ttab x>0
\end{equation} 
are called \bf{Gauss-like} and \bf{Renyi-like} for $m=0$ and $m=1$ respectively.\\

\noindent We expand the initial seed $x_0 \in (0,1)$ as an (m,k)-continued fraction using the following iteration process: 
\begin{enumerate}
\item Set $n :=1$.
\item If $x_{n-1} = 0$, write $x_0 = [a_1,...,a_{n-1}]_{(m,k)}$ and exit. 
\item Set  the \bf{reminder} of $x_0$ at time $n$ to be $r_n := A_{(m,k)}(x_{n-1}) \in (0,\infty)$ and write the (m,k)-CF expansion for $x_0$ at time $n$ as $x_0=[r_1]_{(m,k)}$ if $n=1$ or $x_0 = [a_1,...,a_{n-1},r_n]_{(m,k)}$ if $n > 1$. Also, set the \bf{digit} and \bf {future} of $x_0$ at time $n$ to be 
\begin{equation}\label{a_n}
a_n := \lfloor r_n \rfloor  = \lfloor A_{(m,k)}(x_{n-1}) \rfloor = \bigg\lfloor  \frac{k(1-m-x_{n-1})}{x_{n-1}-m} \bigg\rfloor \in \mathbb{Z}^+ := \mathbb{Z} \cap [0,\infty), 
\end{equation}
and $x_n := [r_{n+1}]_{(m,k)} = r_n - a_n \in [0,1)$. Increase $n$ by one and goto step 2. 
\end{enumerate}
For all $n \ge 0$, we thus have
\[x_{n+1} = T_{(m,k)}(x_n) =  \frac{k(1-m-x_n)}{x_n-m} - a_n\]
so that
\begin{equation}\label{x_n}
x_n = [a_{n+1},r_{n+2}]_{(m,k)} =  m + \frac{k(1-2m)}{a_{n+1}+k+[r_{n+2}]_{(m,k)}} = m + \frac{k(1-2m)}{a_{n+1}+k+x_{n+1}}. 
\end{equation} 
Therefore, this iteration scheme leads to the expansion of the initial seed $x_0$ as
\[x_0 = m + \frac{k(1-2m)}{a_1+k+x_1} = m + \frac{k(1-2m)}{a_1 + k + m + \frac{k(1-2m)}{a_2+k+x_2}} = ...\] 
\begin{remark}\label{digit_remark}
The special case $k=1$ corresponds with the classical Gauss and Renyi maps for $m=0$ and $m=1$ respectively, but with digits that are smaller by one than their classical representation. For instance, 
\[[0,1,2]_{(0,k)} = \dfrac{k}{0+k+\dfrac{k}{1+k+\dfrac{k}{2+k}}} = \frac{k^2 + 4k + 2}{k^2 + 5k + 4}\]
and
\[[0,1,2]_{(1,k)} = 1 - \dfrac{k}{0+k+1-\dfrac{k}{1+k+1-\dfrac{k}{2+k}}} = \frac{k+4}{k^3 + 3k^2 + 5k + 4}\]
will yield, after plugging $k=1$, the fractions $[1,2,3]_0 = \frac{7}{10}$ and $[1,2,3]_1 = \frac{5}{13}$. We label the digits of the $m$-expansion $b_n$ and the (m,k)-expansion $a_n = b_n-1$ to help avoid this confusion.  
\end{remark}

\noindent We call a real numbers in the interval, for which this process terminates by the N$^{\operatorname{th}}$ iteration, an \bf{(m,k)-rational of rank N}. Denote the set of all (m,k)-rationals by $\mathbb{Q}^{(N)}_{(m,k)}$, that is 
\[\mathbb{Q}^{(N)}_{(m,k)} := \big\{x \in [0,1): T^n_{(m,k)}(x) = 0 \tab \text{ for some $n \le N$} \big\}.\]
We further define the set of \bf{(m,k)-rationals} to be $\mathbb{Q}_{(m,k)} := \displaystyle{\lim_{n \to \infty}}\mathbb{Q}^{(n)}_{(m,k)}$ and the set of \bf {(m,k)-irrationals} to be their complement in the interval. Then $x_0 \in \IQ_{(m,k)}$ if and only if $x_0=0$ or $x_0$ has a finite (m,k)-expansion, that is, there exist a unique finite sequence of digits $\{a_n\}_1^N$ such that $x_0 =[a_1,a_2,...,a_N]_{(m,k)}$. We also define the \bf{interval of monotonicity} (or cylinder set) of rank $N \ge 0$ associated with the finite sequence of $N$ non-negative integers $\{a_1,...,a_N\}$ to be $\Delta^{(0)} := (0,1)$ and $\Delta^{(N)}_{a_1,...,a_N} := \big\{x_0 \in (0,1): a_n(x_0) = a_n \tab \text{for all $1 \le n \le N$}\big\}$. Then the restriction of $T_{(m,k)}^N$ to the interior of any interval of monotonicity of rank $N$ is a homeomorphism onto $(0,1)$ and for all $N \ge 0$ we have
\[(0,1) = \displaystyle{\bigcup_{a_1,...,a_N \in \mathbb{Z}^+}}\Delta^{(N)}_{a_1,...,a_N},\] 
where this union is disjoint in pairs.\\

\subsection{Approximation coefficients for Gauss-like and Renyi-like maps}{}

Fix $m \in \{0,1\}, \tab k \in [1,\infty)$. The (m,k)-sequence of approximation coefficients $\big\{\theta_n(x_0)\big\}_0^\infty$ for the (m,k)-expansion is defined just like the classical object 
\begin{equation}\label{theta_x_0}
\theta_n(x_0) := q_n^2\abs{x_0 - \frac{p_n}{q_n}}, \hspace{1pc} n \ge 1,
\end{equation}
where the (m,k)-rational numbers $\frac{p_0}{q_0} = \frac{0}{1}$ and $\frac{p_n}{q_n} = [a_1,...,a_n]_{(m,k)}$ are the corresponding convergents for $x_0$. We further define the \bf{past} of $x_0$ at time $n \ge 0$ to be $Y_0 := m - k, \tab Y_1 := m-k - a_1\in (-\infty,m-k]$ and 
\begin{equation}\label{Y_n}
Y_n := m - k - a_N - [a_{N-1},...,a_1]_{(m,k)} \in (-\infty,m-k), \hspace{1pc} n \ge 2.
\end{equation}

\noindent The sequence of approximation coefficients relates to the future and past sequences of $x_0$ using the identity  
\begin{equation}\label{theta_future_past}
\theta_{n-1}(x_0) = \dfrac{1}{x_n - Y_n}, \hspace{1pc} n \ge 1,
\end{equation}
which was first proved for the classical Gauss case $m=0, \tab k=1$ in 1921 by Perron \cite{Perron}. When $k > 1$ and for all $n >0$, the pair of approximation coefficients $\big(\theta_{n-1}(x_0), \theta_n(x_0)\big)$, also known as the \bf{Jager pair} of $x_0$ at time $n$, lies within the quadrangle in the Cartesian plane with vertices $(0,0), \tab \big(\frac{1}{k},0\big), \tab \big(0,\tab\frac{1}{k}\big)$ and $\big(\frac{1}{k+1-2m},\tab\frac{1}{k+1-2m}\big)$. Note that for the classical Gauss case $m=0, \tab k=1$, this quadrangle degenerates to the triangle with vertices $(0,0), \tab (1,0)$ and $(1,0)$ and for the classical Renyi case $m=k=1$, this quadrangle expands to the infinite region in the first quadrant of the $uv$-plane bounded between the lines $u-v=1$ and $v-u=1$. Conclude that for $x_0 \in (0,1) - \IQ_{(m,k)}, \tab k \ge 1$ and $n > 0$, we have  
\begin{equation}\label{theta_bound_1}
k{\theta_{n-1}(x_0)} +(1-2m)\theta_n(x_0) \le 1,
\end{equation}
and
\begin{equation}\label{theta_bound_2}
{\theta_{n-1}(x_0)} + (1-2m)k\theta_n(x_0) \le 1.
\end{equation}

\subsection{The natural extension}

Fixing $m \in \{0,1\}$ and $k \in [1,\infty)$, the maps $T_{(m,k)}$ both invariant and ergodic with respect to the measures, whose densities on the interval are 
\[\mu_{(m,k)}(x) := \left(\ln\left(\frac{k+1-m}{k-m}\right)(x+k-1)\right)^{-1}.\] 
The induced dynamical systems $\{(0,1),\mathcal{L},\mu,T\}_{(m,k)} := \big\{(0,1) - \mathbb{Q}_{(m,k)}, \mathcal{L}, \mu_{(m,k)}, T_{(m,k)}\big\}$, where $\mathcal{L}$ is the Lebesgue $\sigma$-algebra, are not invertible since the maps $T_{(m,k)}$ are not bijections. However, there is a canonical way to extend non-invertible dynamical systems to invertible ones \cite{Rohlin}. The realization for the natural extension we are about to present was originally introduced to the special case $m=0, \tab k=1$ by Nakada in \cite{Nakada} and plays a vital role in the proof of the Doblin-Lenstra conjecture \cite{BJW}.\\

\noindent Define the region $\Omega' _{(m,k)} := [0,1) \times (-\infty,m-k]$, the set
\[\IQ'_{(m,k)} := \big\{m - k - b - q : b \in \mathbb{Z}^+ \tab \text{and} \tab q \in \IQ_{(m,k)} \big\} \subset (-\infty, m-k] \]
and the \bf{space of dynamic pairs}
\begin{equation}\label{Omega}
\Omega_{(m,k)} := \Omega'  - \left([0,1) \times \IQ'_{(m,k)}\right) \cup \left(\IQ_{(m,k)} \times (-\infty,m-k]\right).
\end{equation}  
The \bf{natural extension map} $\ccT_{(m,k)}:\Omega_{(m,k)} \to \Omega_{(m,k)}$ is defined as $\ccT_{(m,k)}(x,y) =$ 
\[\left(A_{(m,k)}(x) - \lfloor A_{(m,k)}(x) \rfloor, \tab A_{(m,k)}(y) - \lfloor A_{(m,k)}(x) \rfloor\right) = \left(T_{(m,k)}(x), \tab A_{(m,k)}(y) - \lfloor A_{(m,k)}(x) \rfloor\right).\] 
After using the definition \eqref{T} of $T_{(m,k)}$, this map is written explicitly as
\begin{equation}\label{ccT}
\ccT_{(m,k)}(x,y) = \left(\frac{k(1-m-x)}{x-m} - \bigg\lfloor \frac{k(1-m-x)}{x-m} \bigg\rfloor, \tab  \frac{k(1-m-y)}{y-m} - \bigg\lfloor \frac{k(1-m-x)}{x-m} \bigg\rfloor \right).
\end{equation} 
\noindent The maps $\ccT_{(m,k)}$ are both invariant and ergodic with respect to the probability measures $\rho_{(m,k)}(D) := \ln\big(\frac{k+1-m}{k - m}\big)^{-1}\iint_D\frac{dxdy}{(x-y)^2}$ when $k>m$ and the infinite measure $\rho_{(1,1)}(D) := \iint_D\frac{dxdy}{(x-y)^2}$ for the classical Renyi case $m=k=1$ (since there is no finite invariant measure for $T_1$, there is also no finite invariant measure for $\ccT_{(1,1)}$, see remark \eqref{m=k=1}). Furthermore, the dynamical system $\{(0,1),\mathcal{L},\mu,T\}_{(m,k)}$ is realized as a left factor to the invertible dynamical system $\big\{\Omega_{(m,k)},\mathcal{L}^2,\rho_{(m,k)},\ccT_{(m,k)}\big\}$. From now on, we will require the parameter $k$ to be grater than or equal to one and leave the known pathologies of the $0 < k < 1$ cases for a different time (for more information about these cases, refer to \cite{Thesis}).\\

\noindent For the given parameters $m \in \{0,1\}$ and $k \ge 1$, and an initial seed $(x_0,y_0) \in \Omega_{(m,k)}$, we let $\{a_n\}_1^\infty$ be the unique sequence of non-negative integers and $\{r_n\}_1^\infty$ be the unique sequence of remainders such that 
\[x_0 = [r_1]_{(m,k)} = [a_1,r_2]_{(m,k)} = [a_{1},a_{2},r_3]_{(m,k)} = ...\] 
Since $y_0 < m - k$, there exists a unique non-negative integer $a_0$ such that $m - k - a_0 - y_0 \in (0,1)$. Also, from the definition \eqref{Omega} of $\Omega$, we see that this number is an (m,k)-irrational, hence we take $\{a_n\}_{-1}^{-\infty}$ to be the unique sequence of non-negative integers and $\{s_n\}_0^{-\infty}$ be the unique sequence of remainders such that  
\begin{equation}\label{y_n}
m - k - a_0 - y_0 = [s_0]_{(m,k)} = [a_{-1},s_{-1}]_{(m,k)} =  [a_{-1},a_{-2},s_{-2}]_{(m,k)} = ... 
\end{equation}
Using formulas \eqref{a_n} and the definition \eqref{ccT} of $\ccT$, we see that for all $n \in \IZ$, we have
\begin{equation}\label{x,y_n+1}
(x_{n+1}, y_{n+1}) = \ccT_{(m,k)}(x_n,y_n) = \left(\frac{k(1-m-x_n)}{x_n - m} - a_{n+1}, \frac{k(1-m-y_n)}{y_n - m} - a_{n+1}\right).
\end{equation}
We now apply formula \eqref{x_n}, to write explicit formula for the inverse map $\ccT^{-1}$ as
\begin{equation}\label{ccT_inv}
(x_n, y_n) = \ccT_{(m,k)}^{-1}(x_{n+1},y_{n+1}) := \left(m + \frac{(1-2m)k}{k + a_{n+1} + x_{n+1}} , m + \frac{(1-2m)k}{k + a_{n+1} + y_{n+1}}\tab \right).
\end{equation}
Since the quantity $x_n$ is no other than the future of $x_0$ at time $n$ when $n \ge 1$, we naturally call $x_n$ and $y_n$ the \bf{future} and \bf{past} of $(x_0,y_0)$ at time $n \in \mathbb{Z}$. The pair $(x_n,y_n) := \ccT^n_{(m,k)}(x_0,y_0)$ is called the \bf{dynamic pair} of $(x_0,y_0)$ at time $n \in \IZ$ and the bi-sequence $\{a_n\}_{-\infty}^\infty$ is called the (m,k)-\bf{digit bi-sequence} for $(x_0,y_0)$.\\

\noindent From a heuristic point of view, the map $\ccT_{(m,k)}$ can be realized as an invertible left shift operator on the infinite (m,k)-digit bi-sequence 
\[ [[...,-a_{n-1}, -a_n \tab | \tab a_{n+1}, a_{n+2}, ...]]_{(m,k)} \overset{\ccT}{\mapsto} [[...,-a_n,  a_{n+1} \tab | \tab a_{n+2}, a_{n+3} ...]]_{(m,k)}.\]  
The vertical line in this symbolic digit representation of the initial seed pair $(x_0,y_0) \in \Omega$ stands for the present time. The map $\ccT$ can be thought of as a tick of a clock, pushing the present one step forward into the future.\\

\section{Dynamic pairs and approximation pairs}

\noindent Taking the hint from formula \eqref{theta_future_past}, we {\it define} the approximation coefficient for $(x_0,y_0)$ at time $n-1 \in \IZ$ to be 
\begin{equation}\label{theta_dynamic}
\theta_{n-1}(x_0,y_0) := \dfrac{1}{x_n - y_n}
\end{equation}
and refer to the bi-sequence $\{\theta_n(x_0,y_0)\}_{-\infty}^\infty$ as the \bf{bi-sequence of approximation coefficients} or \bf{BAC}. Define the continuous map
\begin{equation}\label{Psi}
\Psi_{(m,k)}: \Omega_{(m,k)} \to \IR^2, \hspace{1pc} (x,y) \mapsto \bigg(\dfrac{1}{x-y}, \frac{(m-x)(m-y)}{(2m-1)k(x-y)}\bigg).
\end{equation}
and use formulas \eqref{x,y_n+1} and \eqref{theta_dynamic} to obtain
\begin{equation}\label{Psi_theta} 
\Psi_{(m,k)}(x_n,y_n) = \big(\theta_{n-1}(x_0,y_0), \theta_n(x_0,y_0)\big), \ttab n \in \IZ.
\end{equation}
We denote the image $\Psi_{(m,k)}(\Omega_{(m,k)})$ by $\Gamma_{(m,k)}$ and, in order to ease the notation, suppress the subscripts $\square_{(m,k)}$ from now on. 
\begin{proposition}\label{uniform_bound}
For all $(u,v) \in \Gamma$, we have 
\begin{equation}\label{u_v_bound_1}
k{u} +(1-2m)v \le 1
\end{equation}
and
\begin{equation}\label{u_v_bound_2}
(1-2m)u + k{v} \le 1.
\end{equation}
\end{proposition}
\begin{proof}
We will first assume that $k - m > 0$. Let $(x_0,y_0) \in \Omega_{(m,k)}$ be any point in the preimage of $(u,v)$ under $\Psi$ and let $\{a_n\}_{-\infty}^\infty$ be the digit bi-sequence for its (m,k)-expansion. Letting $Y_n$ and $y_n$ be the past of $x_0$ and $(x_0,y_0)$ at time $n \ge 1$ as in definitions \eqref{Y_n} and \eqref{y_n}, we see that for all $n \ge 2$, both $m + k + a_n - y_n$ and $m + k + a_n - Y_n$ belong to the interval of monotonicity $\Delta_{(m,k)}^{a_{n-1},a_{n-2},...,a_1}$. Since the length (as in the Lebesgue measure) of intervals of monotonicity tends to zero as their depth tends to infinity, we have $(y_n - Y_n) \to 0$ as $n \to \infty$. Since $x_n > 0$ and $y_n < m - k$, we see that the sequence $\{x_n - y_n\}_0^\infty$ is uniformly bounded from below by the positive number $k - m$. Thus, we have
\[\abs{\theta_{n+1}(x_0) - \theta_{n+1}(x_0,y_0)} = \abs{\frac{1}{x_n - Y_n} - \frac{1}{x_n - y_n}} \to 0 \tab \text{ as $n \to \infty$}.\]   
The fact that $\Psi_{(m,k)}$ is continuous allows us to conclude that the Jager pairs for $(x_0,y_0)$ have the same uniform bounds as those of $x_0$, as expressed in the inequalities \eqref{theta_bound_1} and \eqref{theta_bound_2}, which is precisely the result. When $m=k=1$, the continuity of $\Psi$ implies 
\[\Gamma_{(1,1)} = \Psi_{(1,1)}(\Omega_{(1,1)}) = \displaystyle{\lim_{k \to 1+}}\left(\Psi_{(1,k)}\left(\Omega_{(1,k)}\right)\right) =  \displaystyle{\lim_{k \to 1+}}\Gamma_{(1,k)}.\] 
Since the result holds for all $k >1$, it remains true for the classical Renyi case as well.
\end{proof}

\noindent For all $u,v \ge 0$, define the quantity
\begin{equation}\label{D}
D(u,v) = D_{(m,k)}(u,v) :=\sqrt{1+4(2m-1)k{u}v}.
\end{equation} 

\begin{lemma}
The map $\Psi:\Omega \to \Gamma$ is a homeomorphism with inverse:
\begin{equation}\label{Psi_inv}
\Psi^{-1}(u,v) := \bigg(m + \dfrac{1-D(u,v)}{2u}, m-\dfrac{1+ D(u,v)}{2u}\bigg).  
\end{equation}
\end{lemma}

\begin{proof} 
First, we will show that $\Psi$ is a bijection. Since the map $\Psi$ is surjective onto its image $\Gamma$, we need only show injectiveness. Let $(x_1,y_1), (x_2,y_2)$ be two points in $\Omega$ such that $ \Psi(x_1,y_1) = \Psi(x_2,y_2)$, that is
\[ \bigg(\dfrac{1}{x_1-y_1}, \frac{(m-x_1)(m-y_1)}{(2m-1)k(x_1-y_1)}\bigg) = \bigg(\dfrac{1}{x_2-y_2}, \frac{(m-x_2)(m-y_2)}{(2m-1)k(x_2-y_2)}\bigg).\]  
Equate the first term to obtain
\begin{equation}\label{x-y}
x_1 - y_1 = x_2 - y_2 
\end{equation}
and then equate the second term to obtain $(m-x_1)(m-y_1) = (m-x_2)(m-y_2)$. The basic algebraic equality $(\alpha+\beta)^2 - (\alpha-\beta)^2 = 4\alpha\beta$, using $\alpha=m-x_1,\tab \beta=m-y_1$, will now yield
\[\big(2m - (x_1+y_1)\big)^2 - (x_1-y_1)^2 = 4(m-x_1)(m-y_1)\] 
\[= 4(m-x_2)(m-y_2) = \big(2m - (x_2+y_2)\big)^2 -  (x_2-y_2)^2.\] 
Another application of condition \eqref{x-y} reduces the last equation to 
\[\big(2m - (x_1+y_1)\big)^2 = \big(2m - (x_2+y_2)\big)^2.\] 
We use the definition of $\Omega'$ \eqref{Omega} and observe that $x + y \le 2m$ for all $(x,y) \in \Omega \subset \Omega'$, so that we may conclude the equality $x_1 + y_1 = x_2 + y_2$. More applications of condition \eqref{x-y} will first prove that 
\[x_1 = \frac{1}{2}\big((x_1+y_1) + (x_1-y_1)\big) = \frac{1}{2}\big((x_2+y_2) + (x_2-y_2)\big) = x_2\] 
and then that $y_1=y_2$ as well. Therefore, $\Psi$ is an injection. It is left to prove that $\Psi^{-1}$ is well defined and continuous on $\Gamma$ and that it is the inverse from the left for $\Psi$ on $\Gamma$. Given $(u,v) \in \Gamma$, set 
\[(x,y) := \Psi^{-1}(u,v) = \bigg(m + \dfrac{1-D(u,v)}{2u}, m-\dfrac{1+ D(u,v)}{2u}\bigg).\]

\noindent For the Gauss-like $m=0$ case, we see from inequality \eqref{u_v_bound_1} that $\Gamma$ lies on or underneath the line $k{u}+v=1$ in the $u{v}$ plane. The only point of intersection for this line and the hyperbola $4k{u}v=1$ is the point $(u,v) = \big(\frac{1}{2k},\frac{1}{2}\big)$, hence $\Gamma$ must lie on or underneath this hyperbola as well. Conclude that $4k{u}v \le 1$ for all $(u,v) \in \Gamma$, hence $D(u,v)$ and then $x$ and $y$ are real. We use the inequality $k{u} + v \le 1$ again and obtain 
\[D(u,v)^2 = 1 - 4k{u}v \ge  4u^2{k^2} - 4{k}u + 1 = (2{k}u-1)^2.\] 
Conclude that $1 + D(u,v) \ge 2k{u}$ and $y = - \frac{1+D(u,v)}{2u} \le -k$. Next, we observe that $D(u,v) = \sqrt{1-4k{u}v} < 1$, so that $1 - D(u,v) > 0$, which proves $x = \frac{1-D(u,v)}{2u} > 0$. If we further assume by contradiction that $x = \frac{1-D(u,v)}{2u} \ge 1$ then 
\[\left(1-D(u,v)\right)\left(1+D(u,v)\right) \ge 2u\left(1+D(u,v)\right)\] 
so that 
\[4k{u}v = 1 - D(u,v)^2 = \left(1-D(u,v)\right)\left(1+D(u,v)\right) \ge 2u\left(1+D(u,v)\right).\] 
This implies $2k{v}-1>D(u,v) \ge 0$ so that $(2k{v}-1)^2 \ge D(u,v)^2$, hence $4k^2v^2 - 4k{v} + 1 \ge 1-4k{v}$ and $u +k{v} \ge 1$, in contradiction to inequality \eqref{u_v_bound_2}.\\

\noindent For the Renyi-like $m=1$ case, we have $D(u,v) = \sqrt{1+4k{u}v} > 1$ so that $x = 1 + \frac{1 - D(u,v)}{2u} < 1$. The inequality \eqref{u_v_bound_2} yields 
\[D(u,v)^2 = 1 + 4u{k}v < 1 + 4u(1 + u) = (2u+1)^2,\] 
so that $D(u,v) < 2u +1$. Then $\frac{1 - D(u,v)}{2u} > -1$ and $x= 1 + \frac{1-D(u,v)}{2u} > 0$. Also $1+4k{u}v > 4u^2k^2 - 4k{u} +1$, which implies $\sqrt{1+4k{u}v} > 2k{u} - 1$, so that $\frac{1+\sqrt{1+4kuv}}{2u} = 1-y > k$. Then $y=1-\frac{1+D(u,v)}{2u} < 1-k$. Conclude that $(x,y) \in \Omega$ hence $\Psi^{-1}$ is well defined. Also, $\Psi^{-1}$  is clearly continuous on $\Gamma$.\\ 

\noindent We complete the proof by showing that $\Psi^{-1}$ is injective, that is, $\Psi^{-1}\Psi(x,y)=(x,y)$ for all $(x,y) \in \Omega$. We use our definitions for $\Psi$ \eqref{Psi}, for $\Psi^{-1}$ \eqref{Psi_inv} and the fact that $(2m-1)^2 = 1$ whenever $m \in \{0,1\}$ to obtain
\[\Psi^{-1}\Psi(x,y) = \Psi^{-1}\left(\frac{1}{x-y},\frac{(m-x)(m-y)}{(2m-1)k(x-y)}\right) \]
\[=\left(m + \frac{x-y}{2}\left(1-\sqrt{1 + \frac{4(m-x)(m-y)}{(x-y)^2}}\right), \tab m-\frac{x-y}{2}\left(1+\sqrt{1 + \frac{4(m-x)(m-y)}{(x-y)^2}}\right)\right).\]
\[=\left(m+\frac{x-y}{2}\bigg(1-\sqrt{\left(\frac{2m-x-y}{x-y}\right)^2}\bigg) , \tab m - \frac{x-y}{2}\bigg(1 + \sqrt{\left(\frac{2m-x-y}{x-y}\right)^2}\bigg)\right).\]
But since $2m-x-y \ge x-y > 0$ for all $(x,y) \in \Omega$, this allows us to conclude that the last expression simplifies to
\[\left(m+\frac{x-y}{2}\left(\frac{2(x-m)}{x-y}\right) , \tab m - \frac{x-y}{2}\left(\frac{2(m-y)}{x-y}\right)\right) = (x,y)\]
as desired.
\end{proof}

\section{Symmetries in the BAC}

In this section, we reveal an elegant symmetrical structure for the BAC, allowing us to recover it entirely from a pair of consecutive terms. First, we see that the digit $a_{n+1}$ can be determined from both the pairs of approximation coefficients at times $n$ and $n+1$ in precisely the same fashion. We let
\begin{equation}\label{D_n}
D_{(m,k,n)} = D_n := D(\theta_{n-1},\theta_n) = \sqrt{1+4(2m-1)k\theta_{n-1}\theta_n},
\end{equation}
be as in formula \eqref{D}.
\begin{proposition}\label{present_digits}
Let $a_{n+1}$ be the (m,k)-digit at time $n+1$ and $(\theta_{n-1},\theta_n)$ be the (m,k)-pair of approximation coefficients at time $n$ for the initial seed pair $(x_0,y_0) \in \Omega$. Then
\begin{equation}\label{a_n+1}
a_{n+1} = \bigg\lfloor \frac{D_n+1}{2\theta_n} -k \bigg\rfloor = \bigg\lfloor \frac{D_{n+1}+1}{2\theta_n} -k \bigg\rfloor.
\end{equation}
\end{proposition}
\begin{proof}
Using formula \eqref{Psi_theta}, the fact that $\Psi$ is a bijection and the definition \eqref{Psi_inv} of $\Psi^{-1}$, we have
\begin{equation}\label{x_n_y_n}
(x_n,y_n) =  \Psi^{-1}(\theta_{n-1},\theta_n) = \left(m + \dfrac{1-D_n}{2\theta_{n-1}}, m-\dfrac{1+D_n}{2\theta_{n-1}}\right).
\end{equation}
Using formula \eqref{x_n}, the first components in the exterior terms of formula \eqref{x_n_y_n} equate to 
\[a_{n+1} + k + [r_{n+2}] = \frac{(1-2m)k}{x_n-m} = \frac{(1-2m)2k\theta_{n-1}}{1-D_n}.\] 
But since $[r_{n+2}] <1$, we obtain 
\[a_{n+1} = \big\lfloor a_{n+1} + [r_{n+2}] \big\rfloor = \bigg\lfloor \dfrac{(1-2m)2k\theta_{n-1}}{1-D_n} -k \bigg\rfloor = \bigg\lfloor \dfrac{(1-2m)2k\theta_{n-1}(D_n+1)}{1-D_n^2} -k \bigg\rfloor.\] 
After applying the definition \eqref{D_n} of $D_n$, this expression will then simplify to the first equality in formula \eqref{a_n+1}. Using formula \eqref{y_n}, the second components in the exterior terms of formula \eqref{x_n_y_n} equate to 
\[k + a_n + [s_n] = m - y_n = \frac{D_n+1}{2\theta_{n-1}}.\] 
But since $[s_n] <1$, we have
\[a_n = \big\lfloor a_n + [s_n] \big\rfloor =  \bigg\lfloor \frac{D_n+1}{2\theta_{n-1}} - k \bigg\rfloor.\]
Adding one to all indeces will establish the equality of the exterior terms in formula \eqref{a_n+1}, completing the proof.
\end{proof}

\noindent Next, we will derive a formula to extend the BAC from a pair of consecutive terms, which applies to either the future or the past tail. Define the function $g_{(m,k,a)} = g_a: \Gamma \to \IR$,
\begin{equation}\label{g}
g_a(u,v) = u + \frac{D(u,v)}{(1-2m)k}(m+k+a) +\frac{v}{(2m-1)k}(m+k+a)^2.
\end{equation}
We will prove that:
\begin{theorem}\label{theta_pm_1}
Given the initial seed pair $(x_0,y_0) \in \Omega$, let $a_{n+1}$ be the (m,k)-digit at time $n+1$ and $(\theta_{n-1},\theta_n, \theta_{n+1})$ be the (m,k)-approximation coefficients at time $n-1, n$ and $n+1$. Then
\[\theta_{n\pm1} = g_{a_{n+1}}(\theta_{n\mp1},\theta_n).\]
\end{theorem}
\noindent Combining this result with theorem \ref{present_digits} and the definition \eqref{D_n} of $D_n$, allows us to explicitly write $\theta_{n \pm 1}$ in terms of $(\theta_{n \mp 1},\theta_n)$ as 
\[\theta_{n \pm 1} = \theta_{n \mp 1} + \frac{ \sqrt{1+(2m-1)4k\theta_{n\mp1}\theta_n}}{(1-2m)k}\left(m + k + \bigg\lfloor\frac{1 +  \sqrt{1+(2m-1)4k\theta_{n\mp1}\theta_n}}{2\theta_n} -k \bigg\rfloor\right)\] 
\[+ \frac{\theta_n}{(2m-1)k}\left(m + k + \bigg\lfloor \frac{1 +  \sqrt{1+(2m-1)4k\theta_{n\mp1}\theta_n}}{2\theta_n} -k \bigg\rfloor\right)^2.\]
In order to establish this identity, we will first prove that:
\begin{lemma}
Let $(u,v) \in \Gamma$ and let $a :=  \big\lfloor \frac{D(u,v)+1}{2v} -k \big\rfloor$. Then
\[\Psi\ccT\Psi^{-1}(u,v) = \big(v,g_a(u,v)\big)\]
and
\[u = g_a\big(g_a(u,v),v\big).\] 
\end{lemma}
\begin{proof}
Given $(u,v) \in \Gamma$, use the definition \eqref{Psi_inv} of the map $\Psi^{-1}$ and define the pair 
\begin{equation}\label{x_0,y_0_def}
(x_0,y_0) := \Psi^{-1}(u,v) = \left(m + \frac{1-D(u,v)}{2u},m-\dfrac{1+D(u,v)}{2u}\right) \in \Omega.
\end{equation}
After applying the definition \eqref{Psi} of $\Psi$, we have 
\begin{equation}\label{u,v}
(u,v) = \Psi(x_0,y_0) =  \left(\frac{1}{x_0-y_0}, \tab \frac{(m-x_0)(m-y_0)}{(2m-1)k(x_0-y_0)}\right).
\end{equation}
 We also define the pair $(x_1,y_1)$ to be the image of $(x_0,y_0)$ under $\ccT$, which after using its definition \eqref{ccT}, is written as
\[(x_1,y_1) = \left(\frac{k(1-m-x_0)}{x_0-m} - a_1, \tab \frac{k(1-m-y_0)}{y_0-m} - a_1 \right),\]
where $a_1 := \big\lfloor \frac{k(1-m-x_0)}{x_0-m} \big\rfloor$, hence
\begin{equation}\label{x_1-y_1}
x_1 - y_1 = \frac{k(1-m-x_0)}{x_0-m} - \frac{k(1-m-y_0)}{y_0-m} = \frac{(2m-1)k(x_0-y_0)}{(x_0-m)(y_0-m)}.
\end{equation} 
Applying formula \eqref{x_0,y_0_def} and the definition \eqref{D} of $D(u,v)$, allows us to rewrite this pair as
\[(x_1,y_1) = \left(\frac{2(1-2m)k{u}}{1 - D(u,v)} - k - a_1, \tab + \frac{2(2m-1)k{u}}{1 + D(u,v)} -k - a_1 \right)\]
\begin{equation}\label{x_1,y_1}
= \left( \frac{D(u,v)+1}{2v} -k - a_1 ,\tab \frac{D(u,v)-1}{2v} -k - a_1 \right)
\end{equation}
where 
\begin{equation}\label{a_1}
a_1 = a = \bigg\lfloor \frac{D(u,v)+1}{2v} -k \bigg\rfloor
\end{equation}
is as in the hypothesis. Next, use the definition \eqref{Psi} of $\Psi$ and 
set
\[(v',w) := \Psi(x_1,y_1) = \left(\frac{1}{x_1 - y_1},\tab \frac{(m-x_1)(m-y_1)}{(2m-1)k(x_1 - y_1)}\right).\]
Together with equations \eqref{u,v} and \eqref{x_1-y_1}, this implies that $v=v'=(x_1-y_1)^{-1}$. Using this identity with formula \eqref{x_1,y_1} and the definition $\eqref{D}$ of $D$, we obtain that the second component of $\Psi(x_1,y_1)$ is
\[w = \frac{(m-x_1)(m-y_1)}{(2m-1)k(x_1-y_1)} = \frac{v}{(2m-1)k}(m-x_1)(m-y_1)\]
\[ = \frac{v}{(2m-1)k}\left((m+k+a) - \frac{D(u,v)+1}{2v}\right)\left((m+k+a) - \frac{D(u,v)-1}{2v}\right)\]
\[=\frac{v}{(2m-1)k}\left((m+k+a)^2-\frac{D(u,v)}{v}(m+k+a)+\frac{(2m-1)k{u}}{v} \right).\] 
Conclude that
\begin{equation}\label{w}
w = u + \frac{D(u,v)}{(1-2m)k}(m+k+a) +\frac{v}{(2m-1)k}(m+k+a)^2 = g_a(u,v)
\end{equation}
and since 
\[(v,w) = (v',w) = \Psi(x_1,y_1) = \Psi\ccT(x_0,y_0) = \Psi\ccT\Psi^{-1}(u,v) = \ccK(u,v),\] 
this asserts the validity of the first equation in the hypothesis. \\

\noindent To prove the second part, we use the definition \eqref{ccT_inv} of $\ccT^{-1}$ and write 
\[(x_0,y_0) = \ccT^{-1}(x_1,y_1) = \left(m + \frac{(1-2m)k}{k+a+x_1},m + \frac{(1-2m)k}{k+a+y_1}\right)\] 
where $a$ is as in formula \eqref{a_1}. We also use the definition \eqref{Psi_inv} of $\Psi^{-1}$ again to write
\[(x_1,y_1) = \Psi^{-1}(v,w) = \bigg(m + \frac{1-D(v,w)}{2v},m-\frac{1+D(v,w)}{2v}\bigg).\] 
Combining these observations, we obtain that $(x_0,y_0) = $
\begin{equation}\label{x_0,y_0}
\left(m + \frac{2(1-2m){k}v}{2(m+k+a)v + \left(1 - D(v,w)\right)}, \tab m + \frac{2(1-2m){k}v}{2(m+k+a)v - \left(1 + D(v,w)\right)}\right)
\end{equation}
Using the definition of $\Psi$ \eqref{Psi}, we rewrite $(u,v) = \Psi(x_0,y_0)$ as
\[(u,v) = \left(\frac{1}{x_0-y_0},\tab \frac{(2m-1)(m-x_0)(m-y_0)}{k(x_0-y_0)}\right) = \left(\frac{(2m-1)k{v}}{(m-x_0)(m-y_0)} , \tab v\right).\] 
Together with formula \eqref{x_0,y_0} and the definition \eqref{D} of $D$, we obtain
\[u = (2m-1)k{v}\left(\frac{2(2m-1)k{v}}{2(m+k+a)v + \left(1 - D(v,w)\right)}\right)^{-1}\left(\frac{2(2m-1)k{v}}{2(m+k+a)v - \left(1 + D(v,w)\right)}\right)^{-1}\]
\[= \frac{2m-1}{4k{v}}\left(4(m+k+a)^2v^2 - 4(m+k+a){v}D(v,w) - \left(1-D(v,w)^2\right)\right).\]
\[= w + \frac{D(v,w)}{(1-2m)k}(m+k+a) +\frac{(2m-1)v}{k}(m+k+a)^2 = g_a(w,v).\] 
But from formula \eqref{w}, we have $w = g_a(u,v)$ so that this last observation asserts the validity of the second equation in the hypothesis, completing the proof.
\end{proof}
\begin{proof}(of theorem \ref{theta_pm_1})
We have 
\[(\theta_n,\theta_{n+1}) =  \Psi(x_{n+1},y_{n+1}) =  \Psi\ccT(x_n,y_n) = \Psi\ccT{\Psi^{-1}}(\theta_{n-1},\theta_n) = \ccK(\theta_{n-1},\theta_n)\]
and
\[(\theta_{n-1},\theta_n) =  \Psi(x_n,y_n) =  \Psi\ccT^{-1}(x_{n+1},y_{n+1}) = \Psi\ccT^{-1}{\Psi^{-1}}(\theta_n,\theta_{n+1}) = \ccK^{-1}(\theta_n,\theta_{n+1})\] 
After setting $(u,v) := (\theta_{n\pm1},\theta_n)$, the result is obtained at once from the lemma and proposition \ref{present_digits}.
\end{proof}
\begin{corollary}\label{m+k+a_n+1}
Under the same assumption as the theorem, we have
\begin{equation}
m + k + a_{n+1} = \frac{D_n + D_{n+1}}{2(1-2m)\theta_n}.
\end{equation}
\end{corollary}
\begin{proof}
Using the definition \eqref{g} of $g_a$ and the result of the theorem, we write 
\[\theta_{n-1} = g_{a_{n+1}}(\theta_{n+1},\theta_n) = \theta_{n+1} + \frac{D_{n+1}}{k}(m+k+a_{n+1}) +\frac{(2m-1)\theta_n}{k}(m+k+a_{n+1})^2\] 
\[= g_{a_{n+1}}(\theta_{n-1},\theta_n) + \frac{D_{n+1}}{k}(m+k+a_{n+1}) +\frac{(2m-1)\theta_n}{k}(m+k+a_{n+1})^2\]
\[ = \theta_{n-1} + \left(\frac{D_n + D_{n+1}}{k}\right)(m+k+a_{n+1}) + \frac{2(2m-1)\theta_n}{k}(m+k+a_{n+1})^2,  \]
which yields the desired result after the appropriate cancellations and rearrangements.
\end{proof}

\section{The constant bi-sequence of approximation coefficients}{}

For all $m \in \{0,1\}, \tab k \in [1,\infty)$ and $a \in \IZ^+$, define the constants 
\[ \xi_{(m,k,a)} = \xi_a := [\tab \overline{a} \tab]_{(m,k)} = [a,a,...]_{(m,k)}\] 
and
\begin{equation}\label{C_a}
C_{(m,k,a)} = C_a := \dfrac{1}{\sqrt{(m+k+a)^2 + 4(1-2m)k}},
\end{equation}
where we take $C_{(1,1,0)}$ to be $\infty$. Given two non-negative integers $a$ and $b$, it is clear that  
\begin{equation}\label{C_a<C_b}
a \le b \tab \text{ if and only if }\tab  C_b \le C_a
\end{equation}
and that this inequality remains true if we allow $a$ or $b$ to equal $\infty$.
\begin{theorem}\label{theta_constant}
Let $(x_0,y_0) \in \Omega_{(m,k)}$. Write $a_n := a_n(x_0,y_0), \theta_n := \theta_n(x_0,y_0)$ for all $n \in \mathbb{Z}$ and let $a$ be a non-negative integer. Then the following are equivalent:
\begin{enumerate}[(i)]
\item $a_n = a$ for all $n \in \IZ$.
\item $(x_0,y_0) = (\xi_a,m-a-k-\xi_a)$.
\item $\theta_{-1} = \theta_0 = C_{(m,k,a)}$.  
\item $\theta_n = C_{(m,k,a)}$ \hspace{.1pc} for all $n \in \mathbb{Z}$.
\end{enumerate}
\end{theorem}
\begin{proof}
$\\$(i) $\implies$ (ii): follows directly from formulas \eqref{x_n}, \eqref{y_n} and the definition of $\xi_a$.\\

\noindent (ii) $\implies$ (iii): When $x = \xi_a = [\tab \overline{a} \tab]$, we have $a_1(x,y)=a_1(x) = a$. Furthermore, $T$ acts as a left shift operator on the digits of expansion, hence it fixes $\xi_a$. From the definition \eqref{T} of $T$, we have  
\[\xi_a = [\tab \overline{a} \tab] = T([\tab \overline{a} \tab])= T(\xi_a) = \frac{k(1-m-\xi_a)}{\xi_a-m} -a,\] 
so that
\begin{equation}\label{xi_a_quadratic}
\xi_a^2 - (m-k-a)\xi_a + (m{k}-k-m{a}) = 0.
\end{equation}
Using the quadratic formula, we obtain the roots 
\[\frac{1}{2}\left(m-k-a \pm\sqrt{(a+k-m)^2 -4(m{k}-k-m{a}})\right)\]
\[= \frac{1}{2}\left(m-k-a \pm\sqrt{(m+k+a)^2 +4k(1-2m)}\right).\]
Since the smaller root is clearly negative, we have
\[\xi_a =  \dfrac{1}{2}\left(\sqrt{(m+k+a)^2 +4(1-2m)k}+(m-k-a)\right).\]
In tandem with formula \eqref{C_a}, this provides the relationship
\begin{equation}\label{xi_C_a}
C_a = \dfrac{1}{2\xi_a-(m-k-a)}.
\end{equation}
The starting assumption and the definition \eqref{Psi} of $\Psi$ will now yield 
\[\Psi(x_0,y_0) = \Psi\left(\xi_a, m-k-a-\xi_a)\right) = \left(\frac{1}{2\xi_a-(m-k-a)}, \tab \frac{(2m-1)(m-\xi_a)(a + k + \xi_a)}{k\left(2\xi_a-(m-k-a)\right)}\right)\]
\[=\left(C_a, \tab \frac{2m-1}{k}C_a\left(m{a}+m{k} + (m-k-a)\xi_a - \xi_a^2\right)\right) = (C_a,C_a),\] 
where the last equality is obtained from equation \eqref{xi_a_quadratic}. Combining this last observation with formula \eqref{Psi_theta} yields 
\begin{equation}\label{Psi_C_a}
(\theta_{-1},\theta_0) = \Psi(x_0,y_0) = (C_a,C_a),
\end{equation}
which is the desired result.\\

\noindent (iii) $\implies$ (iv): From the definition \eqref{C_a} of $C_a$, we have
\[1 + 4(2m-1)k{C_a^2} = 1 - \frac{4(1-2m)k}{(m+k+a)^2+4(1-2m)k} =  (m+k+a)^2{C_a}^2 \]
We use this observation and the definition of $g_a$ \eqref{g} to conclude that
\begin{equation}\label{g_C_a}
g_a(C_a,C_a) = C_a + \frac{m+k+a}{(1-2m)k}\sqrt{1 + 4(2m-1)kC_a^2} + \frac{(2m-1)C_a}{k}(m+k+a)^2 = C_a.
\end{equation}
If $\theta_{-1} = \theta_0 = C_a$ then, since $\Psi$ is a bijection, formula \eqref{Psi_C_a} implies $(x_0,y_0) = (\xi_a, -a-k-\xi_a)$ and $a_n(x_0,y_0) = a$ for all $n \in \mathbb{Z}$. Theorem \ref{theta_pm_1} and formula \eqref{g_C_a} now prove the equalities
\[\theta_{1} = g_{a_1}(\theta_{-1},\theta_0) = g_a(C_a,C_a) = C_a\] 
and
\[\theta_{-2} = g_{a_1}(\theta_0,\theta_{-1}) = g_a(C_a,C_a) = C_a.\] 
\noindent The proof that $\{\theta_n\}_{-\infty}^\infty = \{C_a\}$ is the indefinite extension of this argument to all $n \in \mathbb{Z}$.\\

\noindent (iv) $\implies$ (i) : If $\theta_n = C_a$ for all $n \in \mathbb{Z}$, then $D_n = \sqrt{1+4(2m-1)k{C_a}^2}$ for all $n \in \mathbb{Z}$. Corollary \ref{m+k+a_n+1} now yields
\[ (m+k+a_{n+1})^2 = \frac{(D_n + D_{n+1})^2}{4\theta_n^2} = \frac{4(1+4(2m-1)k{C_a}^2)}{4C_a^2} = \frac{1}{C_a^2} + 4(2m-1)k. \]
But from formula \eqref{C_a}, we know that $(m+k+a_{n+1})^2 = \frac{1}{C_{a_{n+1}}^2} + 4(2m-1)k$, so that we must have $a_{n+1} = a$ for all $n \in \mathbb{Z}$.
\end{proof}
\begin{corollary}
Let $(x_0,y_0) \in \Omega$ and write $a_n := a_n(x_0,y_0)$ and $\theta_n := \theta_n(x_0,y_0)$ for all $n \in \mathbb{Z}$. Then $\{a_n\}_{-\infty}^\infty$ is constant if and only if $\{\theta_n\}_{-\infty}^\infty$ is constant.
\end{corollary}
\begin{proof}
The necessary condition follows immediately from the previous theorem. Suppose $\{\theta_n\} = \{\theta\}$ is constant and write $D := \sqrt{1+4(2m-1)k\theta^2}$ so that $D = D_0 = D_1$ is as in formula \eqref{D_n}. Then corollary \ref{m+k+a_n+1} yields
\[(m+k+a_1)^2 = \left(\frac{D_0+D_1}{2\theta_0}\right)^2 = \frac{D^2}{\theta^2} = \frac{1}{\theta^2} + 4(2m-1)k.\] 
But from formula \eqref{C_a}, we know that $(m+k+a_1)^2 = \frac{1}{C_{a_1}^2} + 4(2m-1)k$ and since $\theta_n > 0$, we conclude that $\theta = C_{a_1}$. The previous theorem now proves that $a_n(x_0,y_0) = a_1$ for all $n \in \mathbb{Z}$.
\end{proof}

\section{Essential bounds}{}

Thus far, our treatment of the Gauss-like and Renyi-like cases ran along the same line. However, these bi-sequences can be no further apart when it comes to their essential bounds.

\subsection{The Gauss-like cases}{}

In this section, we focus on the Gauss-like case $m=0$. We fix $k \in [1,\infty)$ and omit the subscript $\square_{(0,k)}$ throughout.
\begin{theorem}\label{triple}
Suppose $(x_0,y_0) \in \Omega$. For all $n \in \IZ$ write $a_{n+1} := a_{n+1}(x_0,y_0)$ and $\theta_n := \theta_n(x_0,y_0)$. Then
\[\min\{\theta_{n-1},\theta_n,\theta_{n+1}\} \le C_{a_{n+1}}\]
and 
\[\max\{\theta_{n-1},\theta_n,\theta_{n+1}\} \ge C_{a_{n+1}},\]
where this constant is sharp.
\end{theorem}
\begin{proof}
Assume, by contradiction, that $\min\{\theta_{n-1},\theta_n,\theta_{n+1}\} > C_{a_{n+1}}$. Then using the definition of $C_a$ \eqref{C_a}, we obtain 
\[\min\{\theta_{n-1}\theta_n,\theta_n\theta_{n+1}\} > C_{a_{n+1}}^2 = \frac{1}{(a_{n+1}+k)^2+4k},\] 
hence 
\[D_n{D_{n+1}} \le \max\{D_n^2,D_{n+1}^2\} =  \max\{1-4k\theta_{n-1}\theta_n, 1-4k\theta_n\theta_{n+1}\} < 1-\dfrac{4k}{(a_{n+1}+k)^2+4k}.\] 
We conclude
\begin{equation}\label{max_D^2}
\max\{ D_n{D_{n+1}},D_n^2,D_{n+1}^2\} < 1-\dfrac{4k}{(a_{n+1}+k)^2+4k} = \dfrac{(a_{n+1}+k)^2}{(a_{n+1}+k)^2+4k}.  
\end{equation}
Also, since we are assuming $\theta_n > C_{a_{n+1}}$, we have 
\[\dfrac{1}{4\theta_n^2} < \dfrac{1}{4C_{a_{n+1}}^2} = \dfrac{(a_{n+1}+k)^2 + 4k}{4}.\] 
Using this last observation together with corollary \ref{m+k+a_n+1} and formula \eqref{max_D^2}, we obtain the contradiction 
\[(a_{n+1}+k)^2 = \frac{1}{4\theta_n^2}(D_n^2 + D_{n+1}^2 + 2D_n{D_{n+1}})\]
\[< \frac{(a_{n+1}+k)^2 + 4k}{4}\bigg(\frac{4(a_{n+1}+k)^2}{(a_{n+1}+k)^2+4k}\bigg) =(a_{n+1}+k)^2,\] 
which proves the first inequality in the hypothesis. The proof of the second inequality is the same {\it mutatis mutandis}. Finally, if $(x,y) = \big(\xi_a,-(a+k+\xi_a)\big) \in \Omega$, then we conclude from theorem \ref{theta_constant} that $a_{n+1} = a$ and $\theta_n = C_a$ for all $n \in \mathbb{Z}$. Thus $C_{a_{n+1}}$ is the best possible constant and these inequalities are sharp. 
\end{proof}
\noindent From the inequality \eqref{C_a<C_b}, we have $C_a \le C_0$ for all $a \ge 0$ and, as direct result, conclude that 
\begin{corollary}\label{hurwitz}
Under the same assumptions as the previous theorem, the inequality
\[\theta_n(x_0,y_0) \le C_0 = \frac{1}{\sqrt{k^2 + 4k}}\] 
holds for infinitely many $n$'s. Furthermore, $C_0$ cannot be replaced with any smaller constant. 
\end{corollary}

\subsection{The Renyi-like cases}{}

In this section, we focus on the Renyi-like case $m=1$. We fix $k \in [1,\infty)$ and omit the subscript $\square_{(1,k)}$ throughout.
\begin{lemma}\label{triple_R}
Suppose $(x_0,y_0) \in \Omega$. For all $n \in \mathbb{Z}$ write $a_n := a_n(x_0,y_0)$ and $\theta_n := \theta_n(x_0,y_0)$. If $\theta_n = \max\{\theta_{n-1},\theta_n,\theta_{n+1}\}$ then $\theta_n \le C_{a_{n+1}} = \frac{1}{\sqrt{(a+k-1)^2 + 4a}}$ with equality precisely when $\theta_{n-1} = \theta_n = \theta_{n+1} = C_{a_{n+1}}$. Similarly, if $\theta_n = \min\{\theta_{n-1},\theta_n,\theta_{n+1}\}$ then $\theta_n \ge C_{a_{n+1}}$ with equality precisely when $\theta_{n-1} = \theta_n = \theta_{n+1} = C_{a_{n+1}}$.
\end{lemma}
\begin{proof}
We will only prove the first claim; the proof for the second claim is the same {\it mutatis mutandis}. If $\theta_n = \max\{\theta_{n-1},\theta_n,\theta_{n+1}\}$ then 
\[D_n = \sqrt{1+4k\theta_{n-1}\theta_n} \le \sqrt{1+4k\theta_n^2}\] 
with equality precisely when $\theta_{n-1}=\theta_n$ and 
\[D_{n+1} = \sqrt{1+4k\theta_n\theta_{n+1}} \le \sqrt{1+4k\theta_n^2}\] 
with equality precisely when $\theta_{n+1}=\theta_n$. Conclude that the inequality
\[\frac{1}{4\theta_n^2}\big(D_n^2+D_{n+1}^2 +2D_nD_{n+1} \big) \le \frac{1}{4\theta_n^2}\big(4(1 + 4k\theta_n^2)\big)\] 
must hold and cannot be replaced with equality unless $\theta_{n-1}=\theta_n=\theta_{n+1}$. In this case, theorem \ref{theta_constant} proves that $\theta_{n-1} = \theta_n = \theta_{n+1} = C_{a_{n+1}}$. Otherwise, we may replace the weak inequality with a strict one. If we further assume by contradiction that $\theta_n \ge C_{a_{n+1}}$ then corollary \ref{m+k+a_n+1} and the definition \eqref{C_a} of $C_a$ with $a=a_{n+1}$, yield the contradiction 
\[(a_{n+1}+k+1)^2 = \frac{1}{4\theta_n^2}\big(D_{n-1}^2+D_n^2 +2D_{n-1}D_n \big) < \frac{1}{4\theta_n^2}\big(4(1 + 4k\theta_n^2)\big)\] 
\[= \frac{1}{\theta_n^2} + 4k \le \frac{1}{C_{a_{n+1}}^2} + 4k = (a_{n+1}+k+1)^2,\] 
proving that $\theta_n$ must be strictly smaller than $C_{a_{n+1}}$ as desired. 
\end{proof}
\begin{theorem}\label{thm_mu_bounds_R}
Suppose $(x_0,y_0) \in \Omega$. For all $n \in \IZ$ write $a_n := a_n(x_0,y_0)$ and $\theta_n := \theta_n(x_0,y_0)$. Let $0 \le l \le L \le \infty$ be such that
\[l = \displaystyle{\liminf_{n \in \mathbb{Z}}\{a_n\} \le \limsup_{n \in \mathbb{Z}}\{a_n\}} = L.\]
Then 
\[C_L \le \displaystyle{\liminf_{n \in \mathbb{Z}}\{\theta_n\} \le \limsup_{n \in \mathbb{Z}}\{\theta_n\}} \le C_l,\]
where these inequalities are sharp.
\end{theorem}
\begin{proof}
From our assumption, there exists $N_0 \ge 1$ such that $l \le a_{n+1}(x_0,y_0) \le L$ for all $n \ge N_0$ and for all $n \le 1-N_0$. After using the inequality \eqref{C_a<C_b}, we conclude that
\begin{equation}\label{C_N_0}
C_L \le C_{a_{n+1}} \le C_l \tab \text{for all $n \ge N_0$ and for all $n \le 1-N_0$.} 
\end{equation}
We will first prove the theorem when at least one of the sequences $\{\theta_n\}_{N_0}^\infty$ and $\{\theta_n\}_{1-N_0}^{-\infty}$ is eventually monotone. Then we will show that this inequality holds in general, after proving its validity when neither sequence is eventually monotone. Finally, we will prove that the constants $C_l$ and $C_L$ are the best possible by giving specific examples for which they are obtained.\\

\noindent First, suppose $\{\theta_n\}_{N_0}^{\infty}$ is eventually monotone in the broader sense. Then there exist $N_1 \ge N_0$ for which $\{\theta_n\}_{N_1}^\infty$ is monotone. By proposition \ref{uniform_bound}, this sequences is bounded in $[0,C_0]$, so it must converge to some real number $C \in [0, C_0]$. Thus
\[\displaystyle{\lim_{n \to \infty}D_n := \lim_{n \to \infty}\sqrt{1+4{k}\theta_{n-1}\theta_n} = \sqrt{1+4{k}C^2}}.\]
Using formula \eqref{C_a} and corollary \ref{m+k+a_n+1}, we obtain 
\[\dfrac{1}{C_{a_{n+1}}^2} + 4k = (a_{n+1}+k+1)^2 = \dfrac{1}{4\theta_n^2}\big(D_n^2+D_{n+1}^2 +2D_n{D_{n+1}} \big),\]
so that 
\[\displaystyle{\lim_{n \to \infty}\frac{1}{C_{a_{n+1}}^2} + 4k =\frac{1}{C^2}(1 + 4{k}C^2) = \frac{1}{C^2} + 4k}\] 
and $\displaystyle{\lim_{n \to \infty}C_{a_{n+1}} = C}$. When $k=1$ is the classical Renyi case, both $C$ and $\displaystyle{\lim_{n \to \infty}}D_n$ might equal infinity, implying that $\frac{1}{C^2} + 4k=0$. Since $\big\{C_{a_n}\big\}_{n \in \mathbb{Z}}$ is a discrete set, there must exists a non-negative integer $a$ and a positive integer $N_2 \ge N_1$ such that $\theta_n = C_{a_{n+1}} = C_a = C$ for all $n \ge N_2$. But this implies from theorem \ref{theta_constant} that $\theta_n = C_a$ for all $n \in \mathbb{Z}$. Since $N_2 \ge N_1 \ge N_0$, we use the inequality \eqref{C_N_0} to conclude that $C_L \le \theta_n = C_a \le C_l$ for all $n \in \mathbb{Z}$, which asserts the validity of the hypothesis for this scenario. Proving that the case when $\{\theta_n\}_{1-N_0}^{-\infty}$ is eventually monotone reduces to the constant case is the same {\it mutatis mutandis}.\\

\noindent Now suppose that both the sequences $\{\theta_n\}_{N_0}^\infty$ and $\{\theta_n\}_{1-N_0}^{-\infty}$ are not eventually monotone, in the broader sense. In particular $\{\theta_n\}_{-\infty}^\infty$ is not constant, so that an application of Theorem \ref{theta_constant} yields $\theta_{n-1} \ne \theta_n$ for all $n \in \mathbb{Z}$. Let $N_1 \ge N_0$ be the first time the sequence $\{\theta_n\}_{N_0}^\infty$ changes direction, that is, we either have $\theta_{N_1} = \min\big\{\theta_{N_1-1},\theta_{N_1},\theta_{N_1+1}\big\}$ or $\theta_{N_1} = \max\big\{\theta_{N_1-1},\theta_{N_1},\theta_{N_1+1}\big\}$. We now show that $C_L < \theta_n < C_l$ for all $n \ge N_1$.\\

\noindent Fixing $N \ge N_1$, take $N',N''$ such that $\theta_{N'}$ and $\theta_{N''}$ are the closest local extrema to $\theta_N$ in the sequence $\{\theta_n\}_{N_1}^\infty$ from the left and right. That is, $N_1 \le N' < N < N''$ and we either have 
\[\theta_{N'} < \theta_{N'+1} < ... < \theta_N < \theta_{N+1} < ... < \theta_{N''}\]
and $\theta_{N'} < \theta_{N'-1}, \tab \theta_{N''} > \theta_{N''+1}$ or  
\[\theta_{N'} > \theta_{N'+1} > ... > \theta_N > \theta_{N+1} > ... > \theta_{N''}.\]
and $\theta_{N'} > \theta_{N'-1}, \tab \theta_{N''} < \theta_{N''+1}$. In the first case, applying the previous lemma to $\theta_{N'} = \min\{\theta_{N'-1},\theta_{N'},\theta_{N'+1}\}$ implies $\theta_{N'} > C_{a_{N'+1}}$ and  applying the previous lemma to $\theta_{N''} = \max\{\theta_{N''-1},\theta_{N''},\theta_{N''+1}\}$ implies $\theta_{N''} < C_{a_{N''+1}}$. In the second case, applying the previous lemma to $\theta_{N'} = \max\{\theta_{N'-1},\theta_{N'},\theta_{N'+1}\}$ implies $\theta_{N'} < C_{a_{N'+1}}$ and  applying the previous lemma to $\theta_{N''} = \min\{\theta_{N''-1},\theta_{N''},\theta_{N''+1}\}$ implies $\theta_{N''} > C_{a_{N''+1}}$. But $N'' > N' \ge N_1 \ge N_0$ so that $l \le a_{N'+1}, \tab a_{N''+1} \le L$. We conclude that
\[C_L \le C_{a_{N'+1}} < \theta_{N'} < \theta_N < \theta_{N''} <  C_{a_{N''+1}} \le C_l\]
in the first case and
\[C_L \le C_{a_{N''+1}} < \theta_{N''} < \theta_N < \theta_{N'} <  C_{a_{N'+1}} \le C_l\]
in the second case. In either case $C_L < \theta_N < C_l$ as desired. Similarly, we let $N_2 \ge N_0$ be the first time the sequence $\{\theta_n\}_{1-N_0}^{-\infty}$ changes direction. The proof that $C_L < \theta_n < C_l$ for all $n \le 1 - N_2$ is the same {\it mutatis mutandis}. After setting $N_3 := \max\{N_1,N_2\}$, we conclude that $C_L < \theta_n < C_l$ for all $\abs{n} > N_3$, which asserts the validity of the hypothesis for this scenario as well.\\   

\noindent Finally, we prove that $C_L$ and $C_l$ are the best possible bounds. Clearly, $C_L = 0$ is the best bound when $L = \infty$ and similarly $C_l = 0$ is the best bound when $l = L = \infty$. To prove $C_l$ is the best possible upper bound when $l < \infty$, fix $l \le L < \infty$, define $x_0 = [a_1,a_2,...]_{(1,k)}$ by 
\[a_n := 
\begin{cases}
L & \text{if $\log_2{n}$ is a positive integer}\\
l & \text {otherwise}
\end{cases}\]
and let $y_0$ be its \bf{reflection}, that is, $y_n := 1 - k - a_1 - [a_2,a_3,...]_{(1,k)}$. Then $l \le a_n(x_0,y_0) \le L$ for all $n \in \mathbb{Z}$ and both digits appear infinitely often. If $N \ge 1$ is such that $a_{N}(x_0,y_0) = L$ then 
\[x_{(N + \log_2{N}+1)} = [\displaystyle{\overbrace{l,l,l,l,l,...,l}^{\text{$\log_2{N}$ times}}},r_{(\log_2{N}+1)}]\] and 
\[y_{(N + \log_2{N}+1)} = 1 - k - l - [\displaystyle{\overbrace{l,l,l,l,l,...,l}^{\text{$\log_2{N}-1$ times}}},s_{(\log_2{N}-1)}].\]
Since this occurs for infinitely many $N$, there exists a subsequence $\{n_j\} \subset \mathbb{Z}$ such that $(x_{n_j},y_{n_j}) \to (\xi_l, 1 - k - l - \xi_l)$. Then theorem \ref{theta_dynamic} and formula \eqref{xi_C_a} prove $\theta_{({n_j}+1)}(x_0,y_0) = \frac{1}{x_{n_j}-y_{n_j}} \to \frac{1}{2\xi_l+k+l-1} = C_l$ as $j \to \infty$. Therefore, $C_l$ cannot be replaced with a smaller constant. The proof that $C_L$ cannot be replaced with a larger constant is the same {\it mutatis mutandis}.  
\end{proof}

\section{Back to one sided sequences}{}
We end this paper by quoting these results which apply to the one-sided sequence of approximation coefficients as well. Fix $m \in \{0,1\}, \tab k \ge 1$ and an initial seed $x_0 \in (0,1) - \IQ_{(m,k)}$. Write $a_n := a_n(x_0)$ and $\theta_n = \theta_n(x_0) = \frac{1}{x_n - Y_n}$ for all $n \ge 1$, where $x_n$ and $Y_n$ are the future and past of $x_0$ at time $n$ as in formulas \eqref{x_n} and \eqref{Y_n} and $\theta_n(x_0)$ is as in the definition \eqref{theta_x_0}. Using formulas \eqref{ccT} and \eqref{ccT_inv}, we see that the maps $\ccT$ and $\ccT^{-1}$ are well defined on $(x_n,Y_n)$ and that for all $n \ge 1$, we have $(x_n,Y_n) = \ccT^n(x_0,Y_0)$ . Using Haas' result \eqref{theta_future_past} and the definition of the map $\Psi$ \eqref{Psi}, we see that we also have $\Psi(x_n,Y_n) = (\theta_{n-1},\theta_n)$. Then, the proofs of proposition \ref{present_digits} as well as theorems \ref{theta_pm_1}, \ref{triple} and \ref{triple_R} remain true, after we restrict $n \ge 1$ and replace $y_n$ with $Y_n$. Consequently, these results apply to the one-sided sequences $\{a_n\}_1^\infty$ and $\{\theta_n\}_1^\infty$ for all parameters $k \ge 1$ in both the Gauss-like and Renyi-like cases.\\

\noindent The proof of proposition \ref{present_digits} for the classical one sided Gauss case $m=0, \tab k=1$ was recently published by the author \cite{Bourla}. The first part of the classical Gauss map, one-sided version of theorem \ref{triple} was first proved by Bagemihl and McLaughlin \cite{BM}, as an improvement on a previous result due to Borel \cite[theorem 5.1.5]{DK}, where the symmetric second part is due to Tong \cite{Tong}. As expected, the constant $C_{(0,1,0)}$ is no other than the Hurwitz Constant $\frac{1}{\sqrt{5}}$.  

\section{Acknowledgments}

This paper is a development of part of the author's Ph.D. dissertation at the University of Connecticut. The author has benefited tremendously from the patience and rigor his Ph.D. advisor Andy Haas and would like to thank him for all his efforts. The author would also like to extend his gratitude to Cor Kraaikamp and Matt Papanikolas, whose suggestions had contributed to a more clear and elegant finished product.

\end{document}